\documentclass[12pt]{article}
\usepackage{setspace}
\usepackage{amsmath,amssymb}
\usepackage{mathtools}
\usepackage{mathrsfs}
\usepackage{geometry}
\usepackage{enumitem}
\usepackage{environ}
\usepackage{setspace}
\usepackage{titlesec}
\usepackage[all,cmtip]{xy}
\usepackage{tikz}
\usetikzlibrary{matrix,arrows}

\usepackage[colorlinks=true,citecolor=blue]{hyperref}
\usepackage[amsthm,thmmarks]{ntheorem}

\NewEnviron{prf}[1][]{\begin{proof}[\bf #1Proof]\BODY\end{proof}}{}
\NewEnviron{slt}[1][]{\begin{proof}[\bf #1	解]\BODY\end{proof}}{}
\newtheorem{definition}{Definition}[section]
\newtheorem{theorem}[definition]{Theorem}
\newtheorem{lemma}[definition]{Lemma}
\newtheorem{proposition}[definition]{Proposition}
\newtheorem{remark}[definition]{Remark}

\newtheorem{notation}[definition]{Notation}

\newtheorem{interpretation}[definition]{Interpretation}

\newtheorem*{DML}{Dynamical Mordell--Lang Conjecture (DML Conjecture)}

\newcommand\address[1]{#1}
\newcommand\email[1]{\emph{Email address}: #1}
  
\setlist[enumerate,1]{label=(\roman*)}
\setlist[enumerate,2]{label=(\alph*)}

\title{\textbf{On the DML(1) property for regular endomorphisms of affine spaces: multiplicities of points at infinity}}
\author{She Yang\quad and\quad Aoyang Zheng}
\date{}

\begin{document}
\begin{spacing}{1.25}

\maketitle

\begin{abstract}
Let $f$ be a regular endomorphism of $\mathbb{A}_{\mathbb{C}}^N$ of algebraic degree $d$ and let $f_{\infty}$ be the induced endomorphism of the infinity hyperplane $H_{\infty}$. Suppose that for every periodic point $x_0\in H_{\infty}(\mathbb{C})$ of period $n_0$, the geometric mean of the multiplicities $e_{f_{\infty}}(x_0),\dots,e_{f_{\infty}}(f_{\infty}^{n_0-1}(x_0))$ is strictly less than $d$. Then we show that the dynamical Mordell--Lang conjecture for $f$ holds for curves in $\mathbb{A}_{\mathbb{C}}^N$.
\end{abstract}

\section{Introduction}

In this article, the base field is $\mathbb{C}$ unless otherwise stated. As a matter of convention, every variety is assumed to be integral, but the closed subvarieties can be reducible. Let $X$ be a variety and let $f$ be an endomorphism of $X$. For a point $x\in X$, the orbit $\mathcal{O}_{f}(x)$ is the set $\{f^{n}(x)\mid n\in\mathbb{N}\}$. We denote $\mathbb{N}=\mathbb{Z}_{+}\cup\{0\}$. An arithmetic progression is a set of the form $\{mk+l\mid k\in\mathbb{N}\}$ for some $m,l\in\mathbb{N}$.

The dynamical Mordell--Lang conjecture is one of the core problems in the field of arithmetic dynamics. It was proposed by Ghioca and Tucker in \cite{GT09} and can be stated as follows.

\begin{DML}
Let $f$ be an endomorphism of a quasi-projective variety $X$. Let $V$ be a closed subvariety of $X$ and let $x\in X(\mathbb{C})$ be a point. Then the return set $\{n\in\mathbb{N}\mid f^n(x)\in V(\mathbb{C})\}$ is a finite union of arithmetic progressions.
\end{DML}

There is an extensive literature on various cases of the DML conjecture. Two significant cases are as follows.

\begin{enumerate}
\item
If $f$ is an \'etale endomorphism of $X$, then the DML conjecture holds. See \cite{Bel06} and \cite[Theorem 1.3]{BGT10}.
\item
If $X = \mathbb{A}^2$, then the DML conjecture holds. See \cite{Xie17} and \cite[Theorem 4]{Xie}.
\end{enumerate}

One can consult \cite{BGT16,Xie} and the references therein for further known results.

In this paper, we mainly study the DML conjecture for curves. So we recall the following notion introduced in \cite[Definition 3]{Xie}.

\begin{definition}
For a quasi-projective variety $X$ and an endomorphism $f$ of $X$, we say $(X,f)$ satisfies the \emph{DML(1) property}, if for any closed subcurve $C\subseteq X$ and any point $x\in X(\mathbb{C})$, the return set $\{n\in \mathbb{N}\mid f^n(x)\in C(\mathbb{C})\}$ is a finite union of arithmetic progressions.
\end{definition}

Here ``1" stands for the dimension of the closed subvariety.

We study the DML(1) property of regular endomorphisms of affine spaces.

\begin{definition}\label{regendo}
A \emph{regular endomorphism} of $\mathbb{A}^N$ is a dominant endomorphism $f$ that can be extended to an endomorphism of $\mathbb{P}^N$. Here, we fix the standard coordinate system $[x_0,\dots,x_N]$ on $\mathbb{P}^N$ and naturally regard $\mathbb{A}^N$ as the open subset $\{x_0\neq0\}$ in it. Therefore, the regular endomorphism $f$ can be written as
$$
f(x_1,\dots,x_N)=(f_1(x_1,\dots,x_N)+g_1(x_1,\dots,x_N),\dots,f_N(x_1,\dots,x_N)+g_N(x_1,\dots,x_N)),
$$
in which $f_1(x_1,\dots,x_N),\dots,f_N(x_1,\dots,x_N)$ are homogeneous polynomials of degree $d\geq1$ that have no nonzero common zeros and $g_1(x_1,\dots,x_N),\dots,g_N(x_1,\dots,x_N)$ are polynomials of degree $<d$. We call $d$ the \emph{algebraic degree} of $f$.
\end{definition}

We fix the following notations that will be used throughout this article.

\begin{notation}\label{notation}
In the setting of Definition \ref{regendo}, we denote $H_{\infty}=\mathbb{P}^N\setminus\mathbb{A}^N$ as the infinity hyperplane. The endomorphism of $H_{\infty}$ induced by $f$ is denoted as $f_{\infty}$.

For an irreducible closed subcurve $C\subseteq\mathbb{A}^N$, we denote $\overline{C}$ as its closure in $\mathbb{P}^N$.

Sometimes we also regard $f$ as an endomorphism of $\mathbb{P}^N$. This should not cause any confusion.
\end{notation}

The concept of regular endomorphisms was introduced in \cite{BJ00} and had been studied in several previous works. For example, please see \cite[Subsections 1.4--1.5]{Xie24}, \cite{DFR}, \cite[Subsection 1.2]{Zhong}, and \cite[Subsection 1.2]{JXZ}. The study of the DML conjecture for regular endomorphisms goes back to Xie's groundbreaking work \cite{Xie17}. See also \cite{YZ,YZb}.

Now we state our main result. For a finite morphism $f:X\rightarrow Y$ between schemes and a point $x\in X$, we let the \emph{multiplicity} of $f$ at $x$ be the length of the Artin local ring $\mathcal{O}_{X,x}/f^*\mathfrak{m}_{f(x)}\mathcal{O}_{X,x}$ and denote this quantity as $e_f(x)$. Here $\mathfrak{m}_{f(x)}$ is the maximal ideal of $\mathcal{O}_{Y,f(x)}$.

\begin{theorem}\label{mainthm}
Let $f$ be a regular endomorphism of $\mathbb{A}^N$ of algebraic degree $d$. Suppose that for every periodic point $x_0\in H_{\infty}(\mathbb{C})$ of period $n_0$, the geometric mean of the multiplicities
$$
e_{f_{\infty}}(x_0),\dots,e_{f_{\infty}}(f_{\infty}^{n_0-1}(x_0))
$$
is strictly less than $d$. Then $f$ satisfies the DML(1) property.
\end{theorem}

\begin{remark}\label{rmk1}
\begin{enumerate}
\item
There are many regular endomorphisms satisfying the hypothesis in Theorem \ref{mainthm}. See \cite[Subsection 2.2]{YZ}.

\item
When we were typing this manuscript, A.I. claimed that it can strengthen Theorem \ref{mainthm} by changing the phrase ``strictly less than $d$" into ``not greater than $d$". If so, then this gives a (maybe) new proof of the DML conjecture for regular endomorphisms of $\mathbb{A}^2$. After comprehending the arguments in this article, the reader may realize that the remaining component is exactly ``the DML conjecture for regular endomorphisms of projective surfaces". We will try to comprehend and arrange its proof for this result later.
\end{enumerate}
\end{remark}

By taking \cite[Theorem 1.3(i)]{YZ} and \cite[Theorem 1.3]{YZb} into account, the key point is proving the following proposition.

\begin{proposition}\label{mainprop}
Let $f$ be a regular endomorphism of $\mathbb{A}^N$ of algebraic degree $d$. Let $q\in H_{\infty}(\mathbb{C})$ be a fixed point of $f_{\infty}$ such that $e_{f_{\infty}}(q)<d$. Let $(C_{-n})_{n\geq0}$ be a sequence of irreducible closed subcurves of $\mathbb{A}^N$ which satisfies $f(C_{-n})=C_{-n+1}$ for every $n\geq1$. Suppose that for each $n$, the normalization of $C_{-n}$ is isomorphic to $\mathbb{A}^1$ and the set-theoretic intersection $\overline{C_{-n}}\cap H_{\infty}=\{q\}$. Then there are only finitely many different curves in $(C_{-n})_{n\geq0}$ and hence all of them are $f$-periodic.
\end{proposition}

\begin{remark}
\begin{enumerate}
\item
The interested reader may check that the proof of Proposition \ref{mainprop} remains valid for arbitrary characteristic (algebraically closed) base fields. But in positive characteristic, the consequence of Proposition \ref{mainprop} for the DML problem of regular endomorphisms is limited as Siegel's theorem fails. See \cite[Section 5]{YZ} for more about this.

\item
Contrary to Remark \ref{rmk1}(ii), it is not hard to see that the hypothesis ``$e_{f_{\infty}}(q)<d$" in Proposition \ref{mainprop} is optimal. The conclusion no longer holds when $e_{f_{\infty}}(q)=d$.
\end{enumerate}
\end{remark}

The structure of this article is as follows. We make some preparations in Section \ref{Sec2}, and then prove the main results in Section \ref{Sec3}. The first half of the proof of Proposition \ref{mainprop} investigates the branches of the curves at infinity, which has a close connection to our arguments in \cite{YZ}. We address more about this in the Appendix.

\textbf{Statement on A.I. use.} The proof of Proposition \ref{mainprop} is learned from ChatGPT 6 Astra. We feel excited that in the first half of its proof (i.e. materials in subsection \ref{sec2.2}), A.I. clarifies some feelings that we had in the work \cite{YZ} but did not be able to fully explain. These thoughts are now summarized in the Appendix. On the contrary, we have no idea about how did A.I. come up with the arguments in the second half of its proof (i.e. arguments in subsection \ref{sec3.1}). Let us just watch its show.

We clear up the original text generated by A.I. and type this manuscript.

\section{Preparations}\label{Sec2}

In this section, we make some preparations for the proof of Proposition \ref{mainprop}. In subsection \ref{sec2.1}, we do some length calculations of local algebra. In subsection \ref{sec2.2}, we introduce some semivaluations corresponding to branches of curves at infinity and describe the relationship of them with the local intersection multiplicities of branches with effective divisors (see Lemma \ref{tautlem} and Definition \ref{locintnum}). These semivaluations shall play a crucial role in the proof of Proposition \ref{mainprop}. In subsection \ref{sec2.3}, we make a preparation about the ``convergent subnets" of sequences in compact topological spaces, which is a technical notion that will be used later.

We proceed in the setting of Proposition \ref{mainprop}.

\subsection{Some length computations in local algebra}\label{sec2.1}

We mainly aim to establish these two results in this subsection. Firstly, we have $e_f(q)=de_{f_{\infty}}(q)$. Secondly, the natural map $f^*\colon\widehat{\mathcal{O}_{\mathbb{P}^N,q}}\to\widehat{\mathcal{O}_{\mathbb{P}^N,q}}$ makes the complete local ring $\widehat{\mathcal{O}_{\mathbb{P}^N,q}}$ into a finite $f^*(\widehat{\mathcal{O}_{\mathbb{P}^N,q}})$-module that can be generated by $e_f(q)$ elements. These results are easy for experts familiar with commutative algebra. We include the proofs for completeness.

We fix some notations that will be used repeatedly in the following.

\begin{notation}\label{notation2}
In order to prove Proposition \ref{mainprop}, we can and do assume $q=[0,0,\dots,0,1]$. We abbreviate $e=e_{f_{\infty}}(q)$.

We rename the coordinates $[x_0,\dots,x_N]$ of $\mathbb{P}^N$ into $[u,z_1,\dots,z_N]$ to emphasize the special status of ``$x_0$" and respect the original notation of A.I.. We write $U_N$ as the affine chart $\{z_N\neq0\}$.

We have the following natural isomorphisms that are compatible with each other
$$
A\coloneqq\mathbb{C}[u,z_1,\dots,z_{N-1}]\cong\Gamma(U_N,\mathcal{O}_{\mathbb{P}^N}),
$$
$$
A_{\mathfrak{m}}=\mathbb{C}[u,z_1,\dots,z_{N-1}]_{(u,z_1,\dots,z_{N-1})}\cong\mathcal{O}_{\mathbb{P}^N,q},
$$
$$
R\coloneqq\mathbb{C}[[u,z_1,\dots,z_{N-1}]]\cong\widehat{\mathcal{O}_{\mathbb{P}^N,q}}.
$$
Here $\mathfrak{m}$ denotes that maximal ideal of $A$, and we always do completions of local rings with respect to their maximal ideals. We may identify these isomorphic objects.

The regular endomorphism $f$ can be written as
$$
[u,z_1,\dots,z_N]\mapsto[u^d,f_1(z_1,\dots,z_N)+ug_1(u,z_1,\dots,z_N),\dots,f_N(z_1,\dots,z_N)+ug_N(u,z_1,\dots,z_N)]
$$
where $f_1,\dots,f_N$ are homogeneous of degree $d$ and $g_1,\dots,g_N$ are homogeneous of degree $d-1$. The term ``$z_N^d$" only exists in $f_N$ and does not appear in $f_1,\dots,f_{N-1}$.

For $1\leq i\leq N$, we write $G_i(u,z_1,\dots,z_{N-1})=f_i(z_1,\dots,z_{N-1},1)+ug_i(u,z_1,\dots,z_{N-1},1)$. Then only $G_N$ has a nonzero constant term. Both of the pull-back $\mathbb{C}$-algebra endomorphisms $f^*$ of $\mathcal{O}_{\mathbb{P}^N,q}$ and $\widehat{\mathcal{O}_{\mathbb{P}^N,q}}$ are given by
$$
(u,z_1,\dots,z_{N-1})\mapsto\left(\frac{u^d}{G_N(u,z_1,\dots,z_{N-1})},\frac{G_1(u,z_1,\dots,z_{N-1})}{G_N(u,z_1,\dots,z_{N-1})},\dots,\frac{G_{N-1}(u,z_1,\dots,z_{N-1})}{G_N(u,z_1,\dots,z_{N-1})}\right).
$$

We denote $\varphi=f^*\colon R\to R$.
\end{notation}

The residue field of every Artin local ring appeared in the sequel is isomorphic to $\mathbb{C}$, hence we recall the fact that their lengths are the same as their dimensions as complex linear spaces. 

\begin{lemma}\label{multicalc}
We have $e_f(q)=de$.
\end{lemma}

\begin{proof}
By definition, $e_f(q)$ is the length of the Artin local ring $A_{\mathfrak{m}}/(u^d,G_1,\dots,G_{N-1})$. We can also calculate that $e$ is the length of $A_{\mathfrak{m}}/(u,G_1,\dots,G_{N-1})$. According to \cite[(II, 8.21A(c))]{Har77}, we see that $G_1,\dots,G_{N-1},u^d$ form a regular sequence in $A_{\mathfrak{m}}$. In particular, $u^d$ is not a zero divisor in $S\coloneqq A_{\mathfrak{m}}/(G_1,\dots,G_{N-1})$. As a result, every successive quotient $u^aS/u^{a+1}S$ is isomorphic to $S/uS$ as an $S$-module via multiplication by $u^a$. This finishes the proof.
\end{proof}

Since Artin local rings are complete, we have isomorphisms
$$
A_{\mathfrak{m}}/(u^d,G_1,\dots,G_{N-1})\cong R/(u^d,G_1,\dots,G_{N-1}),\quad A_{\mathfrak{m}}/(u,G_1,\dots,G_{N-1})\cong R/(u,G_1,\dots,G_{N-1})
$$
and the lengths of these rings are $de$ and $e$, respectively.

\begin{lemma}\label{fingen}
The morphism $\varphi\colon R\to R$ is finite. More precisely, $R$ can be generated by $de$ elements as a $\varphi(R)$-module.
\end{lemma}

\begin{proof}
Let $\mathfrak{m}$ be the maximal ideal of $R$ and write $J=\varphi(\mathfrak{m})R=(u^d,G_1,\dots,G_{N-1})$. Then $de$ is the length of $R/J$ and hence is also its dimension as a complex linear space. Find $h_1,\dots,h_{de}\in R$ such that they form a $\mathbb{C}$-basis in $R/J$. Fix an element $h\in R$. We can find sequences of polynomials $(b_{1,n})_{n\geq1},\dots,(b_{de,n})_{n\geq1}\in A^{\mathbb{Z}_+}$ such that
$$
h-\sum\limits_{a=1}^{de}h_a\varphi(b_{a,n})\in J^n\quad\text{and}\quad b_{a,n+1}-b_{a,n}\in\mathfrak{m}^n
$$
hold for every $n$. For $1\leq a\leq de$, let $b_a\in R$ be the limit of the sequence $(b_{a,n})_{n\geq1}$. Then we have $h=\sum\limits_{a=1}^{de}h_a\varphi(b_a)$. This finishes the proof.
\end{proof}

\subsection{Normalized semivaluations on the curve branches}\label{sec2.2}

By a \emph{semivaluation} on a $\mathbb{C}$-algebra $R$, we mean a map $v\colon R\to\mathbb{R}_{\geq0}\cup\{+\infty\}$ such that
\begin{enumerate}
\item
$v(0)=+\infty$ and $v(\mathbb{C}^{\times})=\{0\}$,
\item
$v(xy)=v(x)+v(y)$ for every $x,y\in R$, and
\item
$v(x+y)\geq\min\{v(x),v(y)\}$ for every $x,y\in R$.
\end{enumerate}

In the setting of Proposition \ref{mainprop}, we have a sequence of curves $(C_{-j})_{j\geq0}$. In the following, we associate a semivaluation on $R=\widehat{\mathcal{O}_{\mathbb{P}^N,q}}$ (see Notation \ref{notation2}) for every $j$.

Since the normalization of $C_{-j}$ is isomorphic to $\mathbb{A}^1$, the normalization of $\overline{C_{-j}}$ is isomorphic to $\mathbb{P}^1$. Let $\pi_j\colon\mathbb{P}^1\to\overline{C_{-j}}$ be the normalization map that sends $\infty$ to $q$. Then $\infty$ is the unique element in $\pi_j^{-1}(q)$. We regard $\pi_j$ as a morphism from $\mathbb{P}^1$ to $\mathbb{P}^N$ by composing with the closed immersion $\overline{C_{-j}}\hookrightarrow\mathbb{P}^N$. Then we get a pull-back morphism
$$
\pi_j^*\colon R=\widehat{\mathcal{O}_{\mathbb{P}^N,q}}\to\widehat{\mathcal{O}_{\mathbb{P}^1,\infty}}.
$$

The semivaluation $v_j$ will be defined as a scalar multiple of the function $\mathrm{ord}_{\infty}\circ\pi_j^*$ (see Definition \ref{semival}). We first investigate this function on the subring $A=\mathbb{C}[u,z_1,\dots,z_{N-1}]$ of $R$.

\begin{lemma}\label{tautlem}
Let $P\in A$ be a polynomial. We regard $P$ as a global section of the line bundle $\mathcal{O}(\deg P)$ on $\mathbb{P}^N$, and get the global section $\pi_j^*P$ of the line bundle $\pi_j^*\mathcal{O}(\deg P)$ on $\mathbb{P}^1$. Then $\mathrm{ord}_{\infty}(\pi_j^*P)$ is the vanishing order of $\pi_j^*P$ at $\infty$. Here the first ``$\pi_j^*P$" is an element in $\widehat{\mathcal{O}_{\mathbb{P}^1,\infty}}$ and the second one is that global section.
\end{lemma} 

We leave the tautological proof of this lemma to the reader.

Let $d_j$ be the degree of the curve $\overline{C_{-j}}$ in $\mathbb{P}^N$.

\begin{definition}\label{semival}
For every $j\geq0$, we define the semivaluation $v_j\colon R\to\mathbb{R}_{\geq0}\cup\{+\infty\}$ by
$$
v_j(h)=\frac{\mathrm{ord}_{\infty}(\pi_j^*h)}{d_j}
$$
for $h\in R$.
\end{definition}

We collect some basic facts.

\begin{lemma}\label{basic}
\begin{enumerate}
\item
For $h\in R$, we have $v_j(h)=0$ if and only if $h\in R^{\times}$.

\item
For $P\in A$, we have $v_j(P)\in[0,\deg P]\cup\{+\infty\}$. Moreover, $v_j(P)=+\infty$ if and only if $P$ vanishes on $\overline{C_{-j}}\cap U_N$, where $U_N$ is the affine chart $\{z_N\neq0\}$ (see Notation \ref{notation2}).

\item
We have $v_j(u)=1$.
\end{enumerate}
\end{lemma}

By taking Lemma \ref{tautlem} into account, the proofs of these facts are easy. We only remark about part (iii). We have $\deg\pi_j^*\mathcal{O}(1)=d_j$, and the zero locus of $\pi_j^*u\in H^0(\mathbb{P}^1,\pi_j^*\mathcal{O}(1))$ is exactly $\{\infty\}$. Hence $\mathrm{ord}_{\infty}(\pi_j^*u)=d_j$ and the result follows.

Next, we investigate the relation of $v_j$ and $v_{j+1}$, and deduce a result similar to the ``triangle inequality" in \cite[Proposition 3.15]{YZ}.

For every semivaluation $v$ on $R$ and every ideal $\mathfrak{a}\subseteq R$, we write
$$
v(\mathfrak{a})\coloneqq\min\{v(a)\mid a\in\mathfrak{a}\}.
$$

\begin{lemma}\label{evolve}
\begin{enumerate}
\item
We have $v_{j+1}(\varphi(h))=dv_j(h)$ for $j\geq0$ and every $h\in R$.

\item
Let $\mathfrak{m}$ be the maximal ideal of $R$. Then we have
$$
ev_{j+1}(\mathfrak{m})\geq\min\{1,dv_j(\mathfrak{m})\}
$$
for every $j$. As a result, we have $v_j(\mathfrak{m})\geq\frac{1}{e}$ for all sufficiently large $j$.
\end{enumerate}
\end{lemma}

\begin{proof}
\begin{enumerate}
\item
Since $f(\overline{C_{-(j+1)}})=\overline{C_{-j}}$, we get surjective endomorphisms $g_j\colon\mathbb{P}^1\to\mathbb{P}^1$ such that
$$
\pi_j\circ g_j=f\circ\pi_{j+1}
$$
holds for every $j$. Write $a_j=\deg g_j$. Since $q$ is the unique element in $(f|_{\overline{C_{-(j+1)}}})^{-1}(q)$, we can see that $g_j^{-1}(\infty)=\{\infty\}$ and hence $g_j$ is a polynomial of degree $a_j$. We have $\pi_{j+1}^*\circ\varphi=g_j^*\circ\pi_j^*$ where
$$
\varphi=f^*\colon R\to R,\quad\pi_j^*,\pi_{j+1}^*\colon R\to\widehat{\mathcal{O}_{\mathbb{P}^1,\infty}},\quad g_j^*\colon\widehat{\mathcal{O}_{\mathbb{P}^1,\infty}}\to\widehat{\mathcal{O}_{\mathbb{P}^1,\infty}}
$$
are the pull-back morphisms. As a result, we have $\mathrm{ord}_{\infty}(\pi_{j+1}^*(\varphi(h)))=a_j\cdot\mathrm{ord}_{\infty}(\pi_j^*h)$ for every $h\in R$.

Taking $h=u$ and using the fact that $\varphi(u)=u^dU$ for some $U\in R^{\times}$, we get
$$
d\cdot\mathrm{ord}_{\infty}(\pi_{j+1}^*u)=a_j\cdot\mathrm{ord}_{\infty}(\pi_j^*u).
$$
According to Lemma \ref{basic}(iii), this means $dd_{j+1}=a_jd_j$. Hence we get $v_{j+1}(\varphi(h))=dv_j(h)$ for every $h\in R$.

\item
We have $v_j(\mathfrak{m})=\min\{1,v_j(z_1),\dots,v_j(z_{N-1})\}>0$. According to Notation \ref{notation2}, we have
$$
v_{j+1}(G_i)=v_{j+1}(\varphi(z_i))=dv_j(z_i)
$$
for every $1\leq i\leq N-1$. Recall that the length of the Artin local ring $R/(u,G_1,\dots,G_{N-1})$ is $e$. Therefore, we have $\mathfrak{m}^e\subseteq(u,G_1,\dots,G_{N-1})$. Thus
\begin{align*}
ev_{j+1}(\mathfrak{m})&=v_{j+1}(\mathfrak{m}^e)\geq v_{j+1}((u,G_1,\dots,G_{N-1}))=\min\{1,v_{j+1}(G_1),\dots,v_{j+1}(G_{N-1})\}\\
&=\min\{1,dv_j(z_1),\dots,dv_j(z_{N-1})\}=\min\{1,dv_j(\mathfrak{m})\}.
\end{align*}

Our assumption of Proposition \ref{mainprop} says that $e<d$. Hence we have $v_j(\mathfrak{m})\geq\frac{1}{e}$ for all sufficiently large $j$ as these quantities are strictly positive.
\end{enumerate}
\end{proof}

\begin{remark}\label{slopermk}
From the geometric point of view, the quantity $v_j(\mathfrak{m})$ can more or less be comprehended as the ``logarithmic slope" of the branch of $C_{-j}$ at $q$. Please see the discussions in subsection \ref{sec4.2}.
\end{remark}

\subsection{Convergent subnets in compact topological spaces}\label{sec2.3}

In this subsection, we recall the concept of ``nets" in general topology. This notion is introduced to overcome the annoying fact that compact spaces may \emph{not} be sequentially compact. The materials in this subsection should be well-known to experts. We use \cite{net} as the reference.

A \emph{directed set} is a set with a preorder relation such that every two elements have an upper bound \cite[Definition 12]{net}.

\begin{definition}(\cite[Definitions 13 and 19]{net})\label{nets}
\begin{enumerate}
\item
A \emph{net} in a topological space $X$ is a function $f$ from a directed set $D$ to X. We usually write $f(d)=x_d$ for all $d\in D$ and refer to the net by $(x_d)_{d\in D}$.

\item
A net $(x_d)_{d\in D}$ \emph{converges} to a point $x\in X$ if for every neighborhood $U$ of $x$, there is a $d\in D$ such that $x_e\in U$ for all $e\geq d$.

\item
Let $(x_d)_{d\in D}$ be a net. A \emph{subnet} of $(x_d)_{d\in D}$ is a net $(x_{f(e)})_{e\in E}$ where $E$ is a directed set and $f\colon E\to D$ is a function such that
\begin{enumerate}
\item
if $e_1\leq e_2$, then $f(e_1)\leq f(e_2)$ (i.e. $f$ is order-preserving) and
\item
for all $d\in D$, there is an $e\in E$ such that $f(e)\geq d$.
\end{enumerate}
\end{enumerate}
\end{definition}

A sequence in a topological space is naturally a net. We need the following result.

\begin{proposition}(\cite[Proposition 20]{net})\label{convergesubnet}
In a compact topological space, every sequence has a convergent subnet.
\end{proposition}

\section{Proof of the main results}\label{Sec3}

We prove Proposition \ref{mainprop} in subsection \ref{sec3.1} and prove Theorem \ref{mainthm} in subsection \ref{sec3.2}.

\subsection{Proof of Proposition \ref{mainprop}}\label{sec3.1}

We proceed from the preparations that have been done in subsections \ref{sec2.1}--\ref{sec2.2}. One can check that the conclusion of Proposition \ref{mainprop} is equivalent to the seemingly weaker assertion that the curves $(C_{-j})_{j\geq0}$ cannot be pairwise distinct. So by contradiction, we assume that they are pairwise distinct. We will deduce a contradiction in this subsection and hence prove Proposition \ref{mainprop}.

For every $j\geq0$, we have introduced a semivaluation $v_j$ on the complete local ring $R$ in subsection \ref{sec2.2}. In the next step, we want to find an increasing sequence $(n_k)_{k\geq0}$ of integers and a sequence $(w_i)_{i\geq0}$ of semivaluations on $R$ such that
$$
w_i(h)=\lim\limits_{k\to\infty}v_{n_k+i}(h)
$$
holds for every $i$ and every $h\in R$. Here we use the natural compact topology on the range $\mathbb{R}_{\geq0}\cup\{+\infty\}$ and abbreviate this topological space as $[0,+\infty]$. However, one may encounter difficulties when finding such a \emph{sequence}. We use the notion of nets in subsection \ref{sec2.3} to overcome such difficulties and proceed as follows. A.I. pointed out that a variant of this argument was used in \cite{BFJ08}.

Consider the compact topological space $X=\prod\limits_{i\geq0,h\in R}[0,+\infty]$ and the sequence $(x_n)_{n\geq0}$ in $X$ defined as $x_n=(v_{n+i}(h))_{i,h}$. Using Proposition \ref{convergesubnet}, we can find a convergent subnet $(x_{n_{\alpha}})_{\alpha\in\Lambda}$ of $(x_n)_{n\geq0}$ where $\Lambda$ is a directed set. Let $w_i(h)$ be the $(i,h)$-component of the limit $\lim\limits_{\alpha\in\Lambda}x_{n_{\alpha}}$. Then we have
$$
w_i(h)=\lim\limits_{\alpha\in\Lambda}v_{n_{\alpha}+i}(h).
$$
For every $i\geq0$, one can check that $w_i$ gives a semivaluation on $R$.

In the following, we always write $\mathfrak{m}$ as the maximal ideal of $R$. According to Lemma \ref{evolve}(ii), we may assume $v_j(\mathfrak{m})\geq\frac{1}{e}$ for every $j\geq0$ after discarding finitely many initial curves in $(C_{-j})_{j\geq0}$. Then the following lemma directly follows from Lemma \ref{basic}(ii) and Lemma \ref{evolve}.

\begin{lemma}\label{wiproperty}
\begin{enumerate}
\item
For every $i$ and every $P\in A$, we have $w_i(P)\in[0,\deg P]\cup\{+\infty\}$.
\item
For $i\geq0$ and $h\in R$, we have $w_{i+1}(\varphi(h))=dw_i(h)$.
\item
We have $w_i(\mathfrak{m})\geq\frac{1}{e}$ for every $i$.
\end{enumerate}
\end{lemma}

Next, we consider the following quantities. For $i\geq0$ and $t\in\mathbb{Z}_+$, we consider the ideal
$$
\mathfrak{a}_t(w_i)\coloneqq\{h\in R\mid w_i(h)\geq t\}
$$
of $R$ and denote $L_i(t)=\dim_{\mathbb{C}}R/\mathfrak{a}_t(w_i)$. A.I. pointed out that such objects had been considered in \cite{ELS03}. Lemma \ref{wiproperty}(iii) guarantees that $\mathfrak{m}^{\lceil et\rceil}\subseteq\mathfrak{a}_t(w_i)$ and hence the numbers $L_i(t)<+\infty$.

The following key proposition gives a lower bound of $L_i(t)$.

\begin{proposition}\label{lowbd}
We have $L_i(t)\geq\frac{t(t+1)}{2}$ for every $i\geq0$ and $t\geq1$.
\end{proposition}

We need the following lemma.

\begin{lemma}\label{dim2}
For $i\geq0$, let $\mathfrak{p}_i\coloneqq\{P\in A\mid w_i(P)=+\infty\}$. Then $\mathfrak{p}_i$ is a prime ideal of the polynomial ring $A$. We have $\dim A/\mathfrak{p}_i\geq2$.
\end{lemma}

\begin{proof}
For every $P\in\mathfrak{p}_i$, we have $w_i(P)=\lim\limits_{\alpha\in\Lambda}v_{n_{\alpha}+i}(P)=+\infty$. Lemma \ref{basic}(ii) then tells us that $P$ vanishes on $\overline{C_{-(n_{\alpha}+i)}}\cap U_N$ for every sufficiently large index $\alpha$. Since the ideal $\mathfrak{p}_i$ is finitely generated, there exists $\alpha_0\in\Lambda$ such that every $P\in\mathfrak{p}_i$ vanishes on the curve $\overline{C_{-(n_{\alpha}+i)}}\cap U_N$ once $\alpha\geq\alpha_0$. We know the function $(n_{\alpha})_{\alpha\in\Lambda}$ is unbounded and we have assumed that the curves $(C_{-j})_{j\geq0}$ are pairwise distinct. Therefore, the zero set defined by $\mathfrak{p}_i$ in $U_N$ contains infinitely many different curves. This finishes the proof.
\end{proof}

Now we can prove Proposition \ref{lowbd}.

\proof[Proof of Proposition \ref{lowbd}]
According to Lemma \ref{dim2}, we may choose two coordinate functions from $u,z_1,\dots,z_{N-1}$, denoted by $a$ and $b$, such that they are algebraically independent over $\mathbb{C}$ in $A/\mathfrak{p}_i$. We claim that the monomials
$$
a^ib^j\quad(i,j\in\mathbb{N},\ i+j<t)
$$
are linearly independent in $R/\mathfrak{a}_t(w_i)$ and hence gain the result. In fact, any nontrivial $\mathbb{C}$-linear combination $P$ of these monomials lies outside $\mathfrak{p}_i$ and has $\deg P<t$. We then get
$$
w_i(P)\leq\deg P<t
$$
by Lemma \ref{wiproperty}(i), which is equivalent to $P\notin\mathfrak{a}_t(w_i)$. This finishes the proof.
\endproof

Finally, we can finish the proof of Proposition \ref{mainprop}.

\proof[Proof of Proposition \ref{mainprop}]
We continue the previous discussions. Take the generators
$$
h_1,\dots,h_{de}\in R
$$
in the proof of Lemma \ref{fingen}. For every $i\geq0$, consider the $\mathbb{C}$-linear (instead of $R$-linear) map $(R/\mathfrak{a}_t(w_i))^{de}\to R/\mathfrak{a}_{dt}(w_{i+1})$ defined by
$$
([b_1],\dots,[b_{de}])\mapsto\left[\sum\limits_{a=1}^{de}h_a\varphi(b_a)\right].
$$
It is well-defined because of Lemma \ref{wiproperty}(ii), and it is by definition surjective. Hence for every $i\geq0$ and every $t\geq1$, we have
$$
L_{i+1}(dt)\leq deL_i(t).
$$
Iterating from $i=0$ and $t=1$ then gives $L_s(d^s)\leq(de)^sL_0(1)$ for every $s\geq0$, while Proposition \ref{lowbd} says that the left-hand-side is bounded below by $\frac{d^s(d^s+1)}{2}$. Since we have assumed that $e<d$, this is absurd for large $s$. This contradiction establishes Proposition \ref{mainprop}.
\endproof

\subsection{Proof of Theorem \ref{mainthm}}\label{sec3.2}

\proof[Proof of Theorem \ref{mainthm}]
Let $x\in\mathbb{A}^N(\mathbb{C})$ be the starting point and let $C\subseteq\mathbb{A}^N$ be an irreducible closed subcurve. Assume that the intersection $\mathcal{O}_f(x)\cap C$ is infinite. We need to prove that $C$ is $f$-periodic.

We first find a sequence of irreducible closed subcurves
$$
\cdots\stackrel{f}\rightarrow C_{-n}\stackrel{f}\rightarrow\cdots\stackrel{f}\rightarrow C_{-1}\stackrel{f}\rightarrow C_0=C
$$
of $\mathbb{A}^N$ such that each $C_{-n}$ has an infinite intersection with the orbit $\mathcal{O}_f(x)$. According to \cite[Theorem 1.3]{YZb}, we may assume that the normalization of each $C_{-n}$ is isomorphic to $\mathbb{A}^1$. Therefore, each set-theoretic intersection $\overline{C_{-n}}\cap H_{\infty}$ contains exactly one point, which we denote as $q_n$. Then we have $f_{\infty}(q_n)=q_{n-1}$ for every $n\geq1$ and \cite[Theorem 1.3(i)]{YZ} says that these $q_n$ are $f_{\infty}$-periodic points. Hence these $q_n$ have to run along a backward $f_{\infty}$-period in $H_{\infty}(\mathbb{C})$. Now the hypothesis about multiplicities, together with Proposition \ref{mainprop}, guarantee that the curve $C$ is $f$-periodic.
\endproof

\section{Appendix: connections with the arguments in \cite{YZ}}

When did the work \cite{YZ}, we got some geometric intuitions about the branches at infinity of curves in $\mathbb{A}^N$. We believe that the results in \cite{YZ} can also be gained in a geometric way by rigorously realizing these intuitions. The paper \cite{YZ} was written in a more arithmetic flavor not only because we like that argument using Roth's theorem, but also due to our technical limitations. Now we have more knowledge about the branches at infinity and we find here an appropriate place to share our thoughts on them. We interpret the idea of \cite{YZ} geometrically and connect them with the quantities in subsection \ref{sec2.2}.

The style of the Appendix may be informal and the arguments may be just heuristic. The main part of this article does not rely on the materials in this Appendix.

\subsection{Branches of curves at infinity}\label{sec4.1}

We consider an irreducible closed curve $C\subseteq\mathbb{A}^N$. Then $\overline{C}\cap H_{\infty}$ is a finite set and we fix a point $q\in\overline{C}\cap H_{\infty}$. There may be more than one branches at a point at infinity, and we want to clarify the concept of ``branches". We find four interpretations of this notion and we believe (or hope) that they are equivalent.

\begin{interpretation}\label{interpret}
\begin{enumerate}
\item
Assume $q=[0,0,\dots,0,1]$ and use the coordinate system $u,z_1,\dots,z_{N-1}$ as in Notation \ref{notation2}. A branch of $C$ at $q$ is defined by $N-1$ equations
$$
z_i=f_i(u),\quad 1\leq i\leq N-1
$$
where $f_1,\dots,f_{N-1}$ are Puiseux series such that this local analytic germ is contained in $C$.

\item
A branch of $C$ at $q$ is a local analytic irreducible component of $C$ at $q$.

\item
Branches of $C$ at $q$ are corresponding to minimal prime ideals in $\widehat{\mathcal{O}_{\overline{C},q}}$.

\item
Let $\pi\colon\widetilde{\overline{C}}\to\overline{C}$ be the normalization map. Branches of $C$ at $q$ are corresponding to points in $\pi^{-1}(q)$.
\end{enumerate}
\end{interpretation}

We make some remarks on these interpretations.

The first interpretation is quite rude. In the two dimensional case, branches can be seen in a down-to-earth way by writing down the Puiseux series. Although this may not be standard in the higher dimensional case, we believe this interpretation should still work. We believe (i) and (ii) are equivalent as the one-parameter germs in (i) should be analytic irreducible. Moreover, every local analytic irreducible component should admit such a presentation because we can project that branch to the coordinate planes $(u,z_i)$ for every $1\leq i\leq N-1$ and then assemble them in reverse. Please notice that there are subtleties in the assembling procedure. For example, the Puiseux series $u^{\frac{3}{2}}$ is not distinct from ``$-u^{\frac{3}{2}}$", while the systems
$$
z_1=u^{\frac{3}{2}},\quad z_2=u^{\frac{3}{2}}
$$
and
$$
z_1=u^{\frac{3}{2}},\quad z_2=-u^{\frac{3}{2}}
$$
are different. We believe that this issue can eventually be clarified anyway.

We believe (ii) and (iii) are equivalent as the minimal primes in the analytic local ring and the complete local ring should admit a one-to-one correspondence. The equivalence of (iii) and (iv) should follow from some algebraic facts \cite[Lemma 33.40.1 (0C1S), Lemma 15.110.7 (0C2D)]{Stacks}.

Sorry but please just regard the paragraphs as heuristics. We did not track through all over it.

\subsection{The logarithmic slope of a branch}\label{sec4.2}

Although it did not explicitly appear in \cite{YZ}, the following concept lies in the kernel of the arguments of that article. For an irreducible closed curve $C\subseteq\mathbb{A}^N$ and a point $q\in \overline{C}\cap H_{\infty}$, we write $C_{\mathrm{br}}(q)$ as a branch of $C$ at $q$.

\begin{definition}\label{logslope}
The \emph{logarithmic slope} of $C_{\mathrm{br}}(q)$ is defined as 
$$
r(C_{\mathrm{br}}(q))\coloneqq\lim\limits_{x\in C_{\mathrm{br}}(q),x\to q}\frac{-\log\mathrm{dist}(\pi(x),q)}{\lambda_{H_{\infty}}(x)}
$$
where $\pi:\mathbb{P}^N\setminus\{O\}\rightarrow H_{\infty}$ is the central projection and $\lambda_{H_{\infty}}(x)$ is the Weil function ``$-\log\mathrm{dist}(x,H_{\infty})$".
\end{definition}

Here the letter ``$r$" stands for ``ratio".

One can comprehend our meaning of ``distance". But rigorously speaking, we need to fix an absolute value to talk about the distance (this issue fits with the setting of \cite{YZ} --- in this article the absolute value is just the standard one on $\mathbb{C}$). We explain that the value $r(C_{\mathrm{br}}(q))$ should be irrelevant to the choice of the absolute value and the limit should exist by using the interpretation \ref{interpret}(i). Indeed, one can see that $r(C_{\mathrm{br}}(q))$ is just the minimal exponent appeared in the expressions of $f_1(u),\dots,f_{N-1}(u)$. In particular, it is strictly positive.

The quantity $r(C_{\mathrm{br}}(q))$ is not quite intrinsic. It equals to $+\infty$ if and only if $C$ is the ``vertical line" passing through $q$ (see \cite[Setting 1.2]{YZ}). The origin $O=(0,\dots,0)\in\mathbb{A}^N$ plays a special role here.

Now let us relate this quantity with ``$v_j(\mathfrak{m})$" in Lemma \ref{evolve}(ii). We give another interpretation of $r(C_{\mathrm{br}}(q))$ according to \ref{interpret}(iv).

We assume $q=[0,0,\dots,0,1]$ and use Notation \ref{notation2}. We write $\pi\colon C_0\to\overline{C}$ as the normalization map and let $p\in\pi^{-1}(q)$ be the point corresponding to $C_{\mathrm{br}}(q)$. Then we get the pull-back map $\pi^*\colon\mathcal{O}_{\mathbb{P}^N,q}\to\mathcal{O}_{C_0,p}$. One can feel that
\begin{align}\label{eq1}
r(C_{\mathrm{br}}(q))=\frac{\min\limits_{1\leq i\leq N-1}\mathrm{ord}_p(\pi^*z_i)}{\mathrm{ord}_p(\pi^*u)}.
\end{align}
Comparing with the expression of $v_j(\mathfrak{m})$ in subsection \ref{sec2.2}, we can find two differences.
\begin{enumerate}
\item
The denominator is $\mathrm{ord}_p(\pi^*u)$ rather than $\deg\overline{C}$.

\item
In the numerator, the term $\mathrm{ord}_p(\pi^*u)$ is missing.
\end{enumerate}

The reason for part (i) is that each curve $C_{-j}$ in subsection \ref{sec2.2} has only one branch at infinity in total and hence we have $\mathrm{ord}_{\infty}(\pi_j^*u)=\deg\overline{C_{-j}}$ in there, and part (ii) gives the identity
$$
v(\mathfrak{m})=\min\{1,r(C_{\mathrm{br}}(q))\}
$$
where $v$ is the semivaluation on $R$ induced by $\pi^*$. To be precise, the semivaluation $v$ is defined by
$$
v(h)=\frac{\mathrm{ord}_p(\pi^*h)}{\mathrm{ord}_p(\pi^*u)},\quad h\in R.
$$

We see that taking the minimum value of $r(C_{\mathrm{br}}(q))$ and $1$ makes this value more intrinsic and compatible with the setting of regular endomorphisms --- only $H_{\infty}$ rather than the origin $O\in\mathbb{A}^N$ plays a role now.

\subsection{The triangle inequality}\label{sec4.3}

The ``triangle inequality" \cite[Proposition 3.15]{YZ} is a crucial step in \cite{YZ}. We formulate a variant of this in the following proposition.

\begin{proposition}\label{triineq}
Let $f$ be a regular endomorphism of $\mathbb{A}^N$ with degree gap $k$ (see \cite[Setting 1.2]{YZ}). Let $C,C'\subseteq\mathbb{A}^N$ be irreducible closed subcurves such that $f(C')=C$. Pick $q\in\overline{C}\cap H_{\infty}$ and $q'\in\overline{C'}\cap H_{\infty}$ such that $f_{\infty}(q')=q$. Let $e=e_{f_{\infty}}(q')$. For a branch $C_{\mathrm{br}}'(q')$ of $C'$ that maps to the branch $C_{\mathrm{br}}(q)$ of $C$, we have
$$
e\cdot r(C_{\mathrm{br}}'(q'))\geq\min\{k,d\cdot r(C_{\mathrm{br}}(q))\}.
$$
\end{proposition}

This proposition is a strengthening of Lemma \ref{evolve}(ii). In \cite[Proposition 3.15]{YZ}, the proof is done by an elementary calculation by using Definition \ref{logslope}. It can also be proved by adapting a similar calculation towards the lengths of local algebras as in subsection \ref{sec2.1}. We leave the details to the interested readers.

\subsection{A fundamental inequality}

Besides the triangle inequality, another key ingredient of \cite{YZ} is the following fact: if every branch $C_{\mathrm{br}}(q)$ of a Siegel curve $C\subseteq\mathbb{A}^N$ at infinity satisfies $r(C_{\mathrm{br}}(q))>2$, then $C$ must be a vertical line \cite[Proposition 4.5]{YZ}. By a Siegel curve, we mean an irreducible curve that its normalization is isomorphic to $\mathbb{A}^1$ or $\mathbb{G}_m$. In other words, it is a curve of genus $0$ with at most two branches at infinity. In fact, the ``genus $0$" condition is not important for the fact above and that fact is a special case of the following ``fundamental inequality".

\begin{proposition}\label{fundineq}
Let $C\subseteq\mathbb{A}^N$ be an irreducible closed subcurve. Then we have
$$
\sum\limits_{\text{branches at infinity}}\frac{1}{r(C_{\mathrm{br}}(q))}\geq1
$$
unless $C$ is a vertical line.
\end{proposition}

In the following, we explain the reason of this inequality. Firstly, we clarify the concept of ``local intersection number" of a branch with an effective Cartier divisor. The following definition is a reformulation of Lemma \ref{tautlem}. We consider a more general setting. The meaning of the notations should be clear.

\begin{definition}\label{locintnum}
Let $X$ be a projective variety and let $C\subseteq X$ be an irreducible closed subcurve. Let $\pi\colon\widetilde{C}\to C$ be the normalization map. Let $q$ be a closed point on $C$, let $C_{\mathrm{br}}(q)$ be a branch of $C$ passing through $q$, and let $p\in\pi^{-1}(q)$ be the point corresponding to the branch $C_{\mathrm{br}}(q)$. Let $L$ be a line bundle on $X$ and let $s$ be a global section of $L$. Then the \emph{intersection number} of $C_{\mathrm{br}}(q)$ with $\mathrm{div}(s)$, denoted as $C_{\mathrm{br}}(q)\cdot\mathrm{div}(s)$, is defined as the coefficient of $[p]$ in the Weil divisor $\mathrm{div}(\pi^*s)$.
\end{definition}

Hence our convention is that $C_{\mathrm{br}}(q)\cdot\mathrm{div}(s)=+\infty$ if and only if $s$ vanishes on $C$. If this is not the case, then we have
$$
\sum\limits_{\text{branches}}(C_{\mathrm{br}}(q)\cdot\mathrm{div}(s))=L\cdot C
$$
where the sum runs over the branches of $C$ at the points in $C\cap\mathrm{Supp}(\mathrm{div}(s))$.

Returning to our original setting about the branches at infinity, the intersection number of $C_{\mathrm{br}}(q)$ with $H_{\infty}$ can be interpreted in the following ways.
\begin{enumerate}
\item
Using the description \ref{interpret}(i), the number $C_{\mathrm{br}}(q)\cdot H_{\infty}$ is the least common denominator of the exponents of the Puiseux series $f_1(u),\dots,f_{N-1}(u)$.

\item
Using the terminology of equation \eqref{eq1}, we have $C_{\mathrm{br}}(q)\cdot H_{\infty}=\mathrm{ord}_p(\pi^*u)$.
\end{enumerate}

We have the following lemma.

\begin{lemma}\label{fanjie}
Suppose the curve $C\subseteq\mathbb{A}^N$ is not a vertical line. Let $C_{\mathrm{br}}(q)$ be a branch of $C$ at infinity. Then we have
$$
r(C_{\mathrm{br}}(q))\cdot(C_{\mathrm{br}}(q)\cdot H_{\infty})\leq\deg\overline{C}
$$
where $\deg\overline{C}$ stands for the degree of $\overline{C}$ in $\mathbb{P}^N$.
\end{lemma}

This lemma is not hard to comprehend. Indeed, the left-hand-side is the numerator in \eqref{eq1}, and hence is the minimum of the intersection numbers $C_{\mathrm{br}}(q)\cdot\mathrm{div}(z_i)$ for $1\leq i\leq N-1$. Here $z_1,\dots,z_{N-1}$ are regarded as global sections of the line bundle $\mathcal{O}(1)$ on $\mathbb{P}^N$. Since $C$ is not a vertical line, these numbers cannot all be $+\infty$. But every finite intersection number $C_{\mathrm{br}}(q)\cdot\mathrm{div}(z_i)$ cannot be greater than $\mathcal{O}(1)\cdot\overline{C}=\deg\overline{C}$. In the setting of \ref{interpret}(i), we can make this explanation down-to-earth by ``reversely solve that $z_i$ from $u$ and write the Puiseux series with independent variable $z_i$".

Now the conclusion of Proposition \ref{fundineq} easily follows by summing the inequality
$$
\frac{C_{\mathrm{br}}(q)\cdot H_{\infty}}{\deg\overline{C}}\leq\frac{1}{r(C_{\mathrm{br}}(q))}
$$
along with the branches at infinity.

After making the arguments up to now rigorous, we can get another geometric way to prove the results in \cite{YZ}.

\subsection{An explanation about \cite[Remark 1.9(i)]{YZ}}

In \cite[Example 1.8(ii)]{YZ}, we consider ``optimal examples" in the setting of loc. cit.. We clarify this notion in the following definition. Please see the Introduction of \cite{YZ} to get the reason why we call these examples ``optimal".

\begin{definition}\label{optex}
Let $f$ be a regular endomorphism of $\mathbb{A}^N$ with degree gap $k$. Let $C\subseteq\mathbb{A}^N$ be an $f$-invariant irreducible closed curve that is not a vertical line. If $k=2\max\limits_{P\in\overline{C}\cap H_{\infty}}e_{f_{\infty}}(P)$, then we say that $(f,C)$ is an \emph{optimal example}.
\end{definition}

In \cite[Remark 1.9(i)]{YZ}, we claim that \cite[Example 1.8(ii)]{YZ} essentially gives all such optimal examples. In this final subsection, we try to use the knowledge in this Appendix to illustrate that remark.

Firstly, we know that $C$ is a Siegel curve. According to Proposition \ref{triineq}, we can see that $r(C_{\mathrm{br}}(q))\geq2$ for every branch of $C$ at infinity. Then Proposition \ref{fundineq} tells us that $C$ has exactly two branches at infinity and the logarithmic slopes of these two branches are both exactly $2$. In particular, the normalization of $C$ is isomorphic to $\mathbb{G}_m$.

Next, we give a down-to-earth description about the logarithmic slopes of branches of Siegel curves.

\begin{lemma}\label{siegelslope}
Let $C\subseteq\mathbb{A}^N$ be a Siegel curve which is not a vertical line.
\begin{enumerate}
\item
Suppose the normalization of $C$ is isomorphic to $\mathbb{A}^1$. Let $\pi\colon\mathbb{A}^1\to\mathbb{A}^N$ be the composition of the normalization map with the closed immersion and write down its formula
$$
\pi(t)=\sum\limits_{j=0}^{a}v_jt^j,\quad v_0,\dots,v_a\in\mathbb{C}^N
$$
where $a>0$ and $v_a\neq0$. Put $b=\max\{j<a\mid v_j\notin\mathbb{C}v_a\}$ (notice that such an element exists). Then the logarithmic slope of the unique branch of $C$ is $\frac{a-b}{a}$.

\item
Suppose the normalization of $C$ is isomorphic to $\mathbb{G}_m$. Let $\pi\colon\mathbb{G}_m\to\mathbb{A}^N$ be the composition of the normalization map with the closed immersion and write down its formula
$$
\pi(t)=\sum\limits_{j=m}^{n}v_jt^j,\quad v_m,v_{m+1},\dots,v_n\in\mathbb{C}^N
$$
where $m<0<n$ and the vectors $v_m$ and $v_n$ are nonzero. Put
$$
b_{+}=\max\{j<n\mid v_j\notin\mathbb{C}v_n\}\quad\text{and}\quad b_{-}=\min\{j>m\mid v_j\notin\mathbb{C}v_m\}.
$$
These elements exist by the same reason as above. Then the logarithmic slopes of the two branches of $C$ are $\frac{n-b_{+}}{n}$ and $\frac{b_{-}-m}{-m}$.
\end{enumerate}
\end{lemma}

One can check this lemma by Definition \ref{logslope}.

Using Lemma \ref{siegelslope}(ii), we can see that the parametrization of the optimal curve $C$ in Definition \ref{optex} has a very restrictive form. More precisely, only the terms $t$ and $t^{-1}$ can appear in the parametrization $\pi(t)$ of $C$. After implementing a linear transformation of $\mathbb{A}^N$, we may assume that $C$ is the curve $\{x_1x_2=1,\ x_3=\cdots=x_N=0\}$ as in \cite[Example 1.8(ii)]{YZ}, i.e. $C$ admits the parametrization $(t,t^{-1},0,\dots,0)$.

Then the formula of $f$ is not hard to determine. We can see that there exist $c\in\mathbb{C}^{\times}$ and a choice of sign such that
$$
f(t,t^{-1},0,\dots,0)=(ct^{\pm d},c^{-1}t^{\mp d},0,\dots,0)
$$
holds. The choice of $ct^d$ and $c^{-1}t^{-d}$ gives exactly the formula in \cite[Example 1.8(ii)]{YZ} and the other choice gives an exactly similar formula.

\section*{Acknowledgements}
We are grateful to our advisor, Junyi Xie, for suggesting us consider the DML(1) property of regular endomorphisms.

This work is supported by the National Natural Science Foundation of China Grant No. 12271007.

\bibliographystyle{alpha}
\bibliography{reference}

@book{BGT16,
    title={The Dynamical Mordell--Lang Conjecture},
    author={J. P. Bell and D. Ghioca and T. J. Tucker},
    publisher={American Mathematical Society},
    year={2016},
    series={Mathematical Surveys and Monographs},
    volume={\textbf{210}},
    address={Providence, RI}
}

@article{Xie17,
    title={The dynamical {M}ordell--{L}ang conjecture for polynomial endomorphisms of the affine plane},
    author={J. Xie},
    journal={Ast\'erisque},
    year={2017},
    volume={\textbf{394}},
    pages={vi+110}
}

@incollection{Xie,
    title={Around the dynamical {M}ordell--{L}ang conjecture},
    author={Junyi Xie},
    BOOKTITLE = {Algebraic, Complex, and Arithmetic Dynamics},
    SERIES = {Simons Symposia},
    PAGES = {59--98},
    PUBLISHER = {Springer Cham},
    YEAR = {2026}
}

@article{Bel06,
    title={A generalized {S}kolem--{M}ahler--{L}ech theorem for affine varieties},
    author={J. P. Bell},
    journal={J. London Math. Soc. (2)},
    year={2006},
    volume={\textbf{73}},
    number={2},
    pages={367--379}
}

@article{BGT10,
    title={The dynamical {M}ordell--{L}ang problem for \'etale maps},
    author={J. P. Bell and D. Ghioca and T. J. Tucker},
    journal={Amer. J. Math.},
    volume={\textbf{132}},
    number={6},
    year={2010},
    pages={1655--1675}
}

@article{GT09,
    title = {Periodic points, linearizing maps, and the dynamical {M}ordell--{L}ang problem},
    author = {D. Ghioca and T.J. Tucker},
    journal = {J. Number Theory},
    volume = {\textbf{129}},
    number = {6},
    pages = {1392--1403},
    year = {2009}
}

@article{Xie24,
    title={Algebraicity criteria, invariant subvarieties and transcendence problems from arithmetic dynamics},
    author={J. Xie},
    journal={Peking Math. J.},
    volume={\textbf{7}},
    number={1},
    pages={345--398},
    year={2024}
}

@unpublished{Zhong,
    title={Polynomial endomorphisms of $\mathbb{A}^2$ with many periodic curves},
    author={X. Zhong},
    note={arXiv:2508.13873v2}
}

@article{DFR,
    title={On the dynamical {M}anin--{M}umford conjecture for plane polynomial maps},
    author={R. Dujardin and C. Favre and M. Ruggiero},
    journal={J. Eur. Math. Soc. (JEMS)},
    year={2026},
    note={https://doi.org/10.4171/jems/1807}
}

@unpublished{JXZ,
    title={Cyclotomic integral points for affine dynamics},
    author={Z. Ji and J. Xie and G.-R. Zhang},
    note={arXiv:2511.13443v2}
}

@article{BJ00,
    title={Dynamics of regular polynomial endomorphisms of $\mathbb{C}^k$},
    author={E. Bedford and M. Jonsson},
    journal={Amer. J. Math.},
    volume={\textbf{122}},
    number={1},
    pages={153--212},
    year={2000}
}

@unpublished{YZ,
    title={Local height arguments toward the dynamical {M}ordell--{L}ang conjecture},
    author={S. Yang and A. Zheng},
    note={arXiv:2605.11676v2}
}

@misc{Stacks,
    title={The {S}tacks project},
    author={Authors Stacks},
    note={https://stacks.math.columbia.edu},
    year={2026}
}

@unpublished{YZb,
    title={On the {DML}(1) property for regular endomorphisms of affine spaces: the $\mathbb{G}_m$-case},
    author={S. Yang and A. Zheng},
    note={arXiv:2609.29107}
}

@book{Har77,
    title={Algebraic Geometry},
    author={R. Hartshorne},
    publisher={Springer-Verlag},
    year={1977},
    series={Graduate Texts in Mathematics},
    volume={\textbf{52}},
    address={New York}
}

@unpublished{net,
    title={Sequences and nets in topology},
    author={S. Vermeeren},
    note={arXiv:1006.4472}
}

@article{BFJ08,
    title={Valuations and plurisubharmonic singularities},
    author={S. Boucksom and C. Favre and M. Jonsson},
    journal={Publ. Res. Inst. Math. Sci.},
    volume={\textbf{44}},
    number={2},
    pages={449--494},
    year={2008}
}

@article{ELS03,
    title={Uniform approximation of {A}bhyankar valuation ideals in smooth function fields},
    author={L. Ein and R. Lazarsfeld and K. E. Smith},
    journal={Amer. J. Math.},
    volume={\textbf{125}},
    number={2},
    pages={409--440},
    year={2003}
}

\address{Beijing International Center for Mathematical Research, Peking University, Beijing 100871, China}

\email{ys-yx@pku.edu.cn}

~

\address{Beijing International Center for Mathematical Research, Peking University, Beijing 100871, China}

\email{zay@pku.edu.cn}

\end{spacing}
\end{document}